\documentclass{amsart}

\newtheorem{theorem}{Theorem}[section]
\newtheorem{corollary}[theorem]{Corollary}
\newtheorem{lemma}[theorem]{Lemma}

\theoremstyle{definition}

\theoremstyle{remark}
\newtheorem{remark}[theorem]{Remark}

\numberwithin{equation}{section}

\begin{document}

\title{Extremal functions of boundary Schwarz lemma}

\author[G. B. Ren]{Guangbin Ren}

\thanks{This work was supported by the NNSF  of China (11071230), RFDP (20123402110068).}

\author[X. P. Wang]{Xieping Wang}
\address{Guangbin Ren, Department of Mathematics, University of Science and
Technology of China, Hefei 230026, China}
\email{rengb$\symbol{64}$ustc.edu.cn}
\address{Xieping Wang, Department of Mathematics, University of Science and
Technology of China, Hefei 230026, China}
\email{pwx$\symbol{64}$mail.ustc.edu.cn}

\keywords{Boundary Schwarz lemma, Extremal function, Osserman inequality.}
\subjclass[2010]{30C35, 32A10}

\begin{abstract}
In this paper, we present an alternative and  elementary proof of a sharp version of the classical  boundary Schwarz lemma by Frolova et al. with initial proof via  analytic semigroup approach and   Julia-Carath\'{e}odory theorem  for univalent holomorphic self-mappings of the open unit disk $\mathbb D\subset \mathbb C$. Our approach has its extra advantage to get the extremal functions of the inequality in the boundary Schwarz lemma.

\end{abstract}
\maketitle

\section{Introduction}
The Schwarz lemma as one of the most influential  results in  complex analysis  puts a great push to the development of several research fields,   such as  geometric function theory, hyperbolic geometry,  complex dynamical systems,  composition operators theory, and theory of quasiconformal mappings. We refer to \cite{Abate, EJLS} for  a more complete insight on the Schwarz lemma.

The classical Schwarz lemma as well as  the Schwarz-Pick lemma concerns with  holomorphic self-mappings of the open unit disk $\mathbb D$ in the complex plane $\mathbb C$, which  provides the invariance of the hyperbolic disks around the interior fixed point under self-mappings of $\mathbb D$. A verity of its boundary versions are in the spirit of Julia \cite{Julia}, Carath\'{e}odory \cite{Caratheodory}, and  Wolff \cite{Wolff1, Wolff2} involving the boundary involving the boundary fixed points.

Recall that a boundary point $\xi\in\partial \mathbb D$ is called a fixed point of $f\in {\rm{\textsf{H}}}(\mathbb D, \mathbb D)$ if
$$f(\xi):=\lim\limits_{r\rightarrow 1^-}f(r\xi)=\xi.$$
Here  ${\rm{\textsf{H}}}(\mathbb D, \mathbb D)$ denotes the class of holomorphic self-mappings of the open unit disk $\mathbb D$. It is well known that, for any  $f\in H(\mathbb D, \mathbb D)$,  its radial limit is the same as its angular limit and both exist for almost all $\xi\in\partial \mathbb D$; moreover, the exceptional set in  $\partial \mathbb D$
is of capacity zero.

The classification of the boundary fixed points of $f\in {\rm{\textsf{H}}}(\mathbb D, \mathbb D)$ can be performed via the value of the \emph{angular derivative}
$$f'(\xi):=\angle\lim\limits_{z\rightarrow \xi}\frac{f(z)-\xi}{z-\xi},$$
which belongs to $(0,\infty]$ due to  the celebrated Julia-Carath\'{e}odory theorem; see \cite{Caratheodory,Abate}. This theorem also asserts that the finite angular derivative at the boundary fixed point $\xi$ exists if and only if the holomorphic function $f'$ admits  the finite angular limit $\angle\lim\limits_{z\rightarrow \xi}f'(z)$. For a boundary fixed point $\xi$ of $f$,
if $$f'(\xi)\in (0,\infty),$$ then
$\xi$ is called a regular  boundary  fixed  point. The regular points can be \emph{attractive} if $f'(\xi)\in (0,1)$, \emph{neutral} if $f'(\xi)=1$, or \emph{repulsive} if $f'(\xi)\in (1,\infty)$.

The  Julia-Carath\'{e}odory theorem \cite{Caratheodory, Abate} and the Wolff lemma \cite{Wolff2} imply that  there exists a unique regular boundary fixed point $\xi$ such that $$f'(\xi)\in (0,1]$$
if $f\in {\rm{\textsf{H}}}(\mathbb D, \mathbb D)$ with no interior fixed point; otherwise the assumption that the mapping $f\in {\rm{\textsf{H}}}(\mathbb D, \mathbb D)$ with an interior fixed point forces  $f'(\xi)>1$
for any boundary fixed point $\xi\in\partial \mathbb D$.
Moreover, Unkelbach \cite{Unkelbach} and  Herzig  \cite{Herzig} proved that if $f\in {\rm{\textsf{H}}}(\mathbb D, \mathbb D)$ has a regular boundary fixed point at point $1$, and $f(0)=0$, then
\begin{equation}\label{UHO}
f'(1)\geq\frac{2}{1+|f'(0)|}.
\end{equation}
Moreover, equality in (\ref{UHO}) holds if and only if $f$ is of the form
$$f(z)=-z\frac{a-z}{1-az},\qquad \forall\,z\in\mathbb D,$$
for some constant $a\in (-1, 0 ]$.

This result is improved sixty  years later by
  Osserman \cite{Osserman} by removing the assumption of the existence of interior fixed points.

\begin{theorem} {\bf(Osserman)}\label{Osserman}
 Let $f\in {\rm{\textsf{H}}}(\mathbb D, \mathbb D)$ with $\xi=1$ as its  regular  boundary  fixed  point. Then  \begin{equation}\label{Osserman-inequality}
 f'(1)\geq\frac{2\big(1-|f(0)|\big)^2}{1-|f(0)|^2+|f'(0)|}.
 \end{equation}
\end{theorem}
This inequality is strengthened very recently by Frolova et al.  in \cite{FLSV} as follows.

\begin{theorem}{\bf(Frolova et al.)}\label{Main theorem}
Let $f\in {\rm{\textsf{H}}}(\mathbb D, \mathbb D)$  with $\xi=1$ as its  regular  boundary  fixed  point. Then
 \begin{equation}\label{key inequality}
 f'(1)\geq\frac{2}{{\rm{Re}}\bigg(\dfrac{1-f(0)^2+f'(0)}{(1-f(0))^2}\bigg)}.
 \end{equation}
 \end{theorem}

The initial proof given in \cite{FLSV} is based on  analytic semigroup  approach as well as  the Julia-Carath\'{e}odory theorem  for {\it univalent} holomorphic self-mappings of $\mathbb D$, which is proved via the method of extremal length \cite{Anderson}.

 The purpose of this article is to study  the extremal functions of  inequality (\ref{key inequality}), in addition to present an alternative and  elementary proof of (\ref{key inequality}). Our main result is as follows.

\begin{theorem}\label{Main theorem-2}
Let $f\in {\rm{\textsf{H}}}(\mathbb D, \mathbb D)$  with $\xi=1$ as its  regular  boundary  fixed  point. Then equality holds in inequality $(\ref{key inequality})$ if and only if $f$ is of the form
\begin{equation}\label{exe-fun}
f(z)=\dfrac{f(0)-z\dfrac{a-z}{1-az}\dfrac{1-f(0)}{1-\overline{f(0)}}}{1-z\dfrac{a-z}{1-az}
\dfrac{1-f(0)}{1-\overline{f(0)}}\overline{f(0)}},\qquad \forall\,z\in\mathbb D,
\end{equation}
for some constant $a\in [-1,1)$.
\end{theorem}

As a direct consequence of Theorems \ref{Main theorem} and \ref{Main theorem-2}, we obtain a strong version of Osserman's inequality.
\begin{corollary}\label{Main corollary}
 Let $f\in {\rm{\textsf{H}}}(\mathbb D, \mathbb D)$  with $\xi=1$ as its  regular  boundary  fixed  point. Then
 \begin{equation}\label{inequality from below}
f'(1)\geq\frac{2\big|1-f(0)\big|^2}{1-|f(0)|^2+|f'(0)|}.
 \end{equation}
 Moreover, equality holds in this inequality if and only if $f$ is of the form
\begin{equation}\label{exe-funs}
f(z)=\dfrac{f(0)-z\dfrac{a-z}{1-az}\dfrac{1-f(0)}{1-\overline{f(0)}}}{1-z\dfrac{a-z}{1-az}
\dfrac{1-f(0)}{1-\overline{f(0)}}\overline{f(0)}},\qquad \forall\,z\in\mathbb D,
\end{equation}
for some constant $a\in [-1,0]$.
\end{corollary}

\begin{remark}\label{Main remark}
From Corollary \ref{Main corollary}, one easily deduce that equality in (\ref{Osserman}) hold if and only if $f$ is of the form
$$f(z)=\dfrac{f(0)-z\dfrac{a-z}{1-az}}{1-z\dfrac{a-z}{1-az}f(0)}, \qquad \forall\,z\in\mathbb D,$$
for some constant $a\in [-1,0]$ with $f(0)\in \mathbb [0,1)$.
\end{remark}

As an application, Corollary \ref{Main corollary}  immediately results in a quantitative strengthening of a classical theorem of L\"{o}wner (i.e  the second assertion in the following Corollary, see \cite{Lowner}) as follows.
\begin{corollary} {\bf(L\"{o}wner)}\label{thm:Lowner}
Let $f\in {\rm{\textsf{H}}}(\mathbb D, \mathbb D)$ and $f(0)=0$. Assume that $f$ extends continuously to an arc $C\in\partial\mathbb D$ of length $s$ and maps it onto an arc $f(C)\in\partial\mathbb D$ of length $\sigma$. Then
$$\sigma\geq\frac{2}{1+|f'(0)|}s.$$
In particular, we have
\begin{equation}\label{Lowner ineq}
\sigma\geq s
\end{equation}
with equality if and only if either $\sigma=s=0$ or $f$ is just a rotation.
\end{corollary}

\begin{remark}
The length $\sigma$ of $f(C)$ is to be taken with multiplicity, if $f(C)$ is a multiple covering of the image.
\end{remark}
For generalizations of Theorems \ref{Main theorem} and \ref{Main theorem-2} as well as  Corollary \ref{Main corollary} to the setting of quaternions for slice regular self-mappings of the open unit ball $\mathbb B\in\mathbb H$, see \cite{RW}.

\section{Proof of main results}
In this section, we shall give the proofs of the main results. Before presenting the details, we first recall the concrete contents of the classical Julia lemma and Julia-Carath\'{e}odory theorem; see e.g. \cite{Abate, Sarason1}, \cite[p. 48 and p. 51]{Sarason2}.

\begin{lemma}{\bf(Julia)}\label{Julia}
Let $f\in {\rm{\textsf{H}}}(\mathbb D, \mathbb D)$ and let $\xi\in\partial \mathbb D$. Suppose that there exists a sequence $\{z_n\}_{n\in \mathbb N}\subset \mathbb D$ converging to $\xi$ as $n$ tends to $\infty$, such that the limits
$$\alpha:=\lim\limits_{n\rightarrow\infty}
\frac{1-|f(z_n)|}{1-|z_n|}$$
and
$$\eta:=\lim\limits_{n\rightarrow\infty}f(z_n)$$
exist $($finitely$)$. Then $\alpha>0$ and the  inequality
\begin{eqnarray}\label{ineq:Julia}
\frac{\big|f(z)-\eta\big|^2}{1-|f(z)|^2}\leq \alpha\,\frac{|z-\xi|^2}{1-|z|^2}
\end{eqnarray}
holds throughout the open unit disk $\mathbb D$ and is strict except for  M\"obius transformations of $\mathbb D$.
\end{lemma}

\begin{theorem}\label{Julia-Caratheodory}{\bf(Julia-Carath\'{e}odory)}
Let $f\in {\rm{\textsf{H}}}(\mathbb D, \mathbb D)$ and let $\xi\in\partial \mathbb D$. Then the following conditions are equivalent:
\begin{enumerate}
\item[(i)]  The lower limit
\begin{eqnarray}\label{def:alpha-Julia}\alpha:=\liminf\limits_{z\rightarrow \xi}\dfrac{1-|f(z)|}{1-|z|}
\end{eqnarray} is finite,
where the limit is taken as $z$ approaches $\xi$ unrestrictedly in $\mathbb D$;

\item[(ii)]  $f$ has a non-tangential limit, say $f(\xi)$, at the point $\xi$, and the difference quotient $$\frac{f(z)-f(\xi)}{z-\xi}$$ has a non-tangential limit, say $f'(\xi)$, at the point $\xi$;
\item[(iii)]  The derivative $f'$ has a non-tangential limit, say $f'(\xi)$, at the point $\xi$.
\end{enumerate}

Moreover, under the above conditions we have
\begin{enumerate}
\item[(a)]  $\alpha>0$ in $(\rm{i})$;

\item[(b)]  the   derivatives $f'(\xi)$ in $(\rm{ii})$ and $(\rm{iii})$ are the same;

\item[(c)]  $f'(\xi)=\alpha \overline{\xi}f(\xi)$;

\item[(d)]  the quotient $\dfrac{1-|f(z)|}{1-|z|}$ has the non-tangential limit $\alpha$ at the point $\xi$.
\end{enumerate}
\end{theorem}

\bigskip
Now we come to prove Theorems \ref{Main theorem} and   \ref{Main theorem-2}.
\begin{proof}[Proofs of Theorems $\ref{Main theorem}$ and  $\ref{Main theorem-2}$]
Let $f$ be as described in Theorem $\ref{Main theorem}$. Set
$$g(z):=\frac{f(z)-f(0)}{1-\overline{f(0)}f(z)}\frac{1-\overline{f(0)}}{1-f(0)},$$ which is in ${\rm{\textsf{H}}}(\mathbb D, \mathbb D)$ such that $\xi=1$ is its regular  boundary  fixed  point and $g(0)=0$. Moreover, an easy calculation shows that
\begin{equation}\label{ineq:01}
f'(1)=\frac{|1-f(0)|^2}{1-|f(0)|^2}\,g'(1),
\end{equation}
and
\begin{equation}\label{ineq:02}
g'(0)=\frac{f'(0)}{1-|f(0)|^2}\frac{1-\overline{f(0)}}{1-f(0)},
\end{equation}
which is no more than one in modulus.
Applying the Julia-Carath\'{e}odory theorem and the Julia inequality (\ref{ineq:Julia}) in the Julia lemma to the holomorphic function
  $h: \mathbb D\rightarrow \overline{\mathbb D}$
defined by
$$h(z
):=\dfrac{g(z)}{z},\qquad \forall\,z\in\mathbb D,$$
we obtain
\begin{equation}\label{derivative ineq}
g'(1)=1+h'(1)\geq 1+\frac{|1-g'(0)|^2}{1-|g'(0)|^2}=\frac{2\big(1-\textrm{Re}g'(0)\big)}{1-|g'(0)|^2}.
\end{equation}
In particular,
\begin{equation}\label{ineq:03}
  g'(1)\geq\frac{2}{1+\textrm{Re}g'(0)}.
\end{equation}
Now inequality  (\ref{key inequality}) follows by substituting equalities in (\ref{ineq:01}) and  (\ref{ineq:02}) into (\ref{ineq:03}).

If equality holds  in inequality (\ref{key inequality}), then equalities also hold in the Julia inequality (\ref{ineq:Julia}) at point $z=0$ and inequality (\ref{ineq:03}), it follows from the condition for equality in the Julia inequality and that for equality in inequality (\ref{ineq:03}) that
 \begin{equation}\label{ineq:04}
 g(z)=z\frac{z-a}{1-\bar{a}z}\frac{1-\bar{a}}{1-a},
\end{equation}
for some constant $a\in\overline{\mathbb D}$, and $g'(0)\in(-1,1]$, which is possible only if $a\in [-1,1)$. Consequently, $f$ must be of the form
\begin{equation}\label{extremal function}
f(z)=\dfrac{f(0)-z\dfrac{a-z}{1-az}\dfrac{1-f(0)}{1-\overline{f(0)}}}{1-z\dfrac{a-z}{1-az}
\dfrac{1-f(0)}{1-\overline{f(0)}}\overline{f(0)}},\qquad \forall\,z\in\mathbb D,
\end{equation}
where $a\in [-1,1)$.
Therefore, the equality in inequality (\ref{key inequality}) can hold only for holomorphic self-mappings  of  the form (\ref{extremal function}), and a direct calculation shows that it does indeed hold for all such holomorphic self-mappings.
This completes the proof.
\end{proof}

\begin{proof}[Proof of Corollary $\ref{Main corollary}$]
Inequality (\ref{inequality from below})  follows immediately from inequality (\ref{key inequality}), and equality in (\ref{inequality from below}) holds if and only if
$$\frac{f'(0)}{\big (1-f(0)\big)^2} \in [0, \infty),$$
which is equivalent to $g'(0)\in \mathbb [0, 1]$, i.e. $a\in [-1,0]$.
Here the function $g$ is the one in (\ref{ineq:04}).

\end{proof}

\begin{proof}[Proof of Corollary $\ref{thm:Lowner}$]
By the classical Schwarz reflection principle, $f$ can extend to be holomorphic in the interior of $C$. Applying Corollary $\ref{Main corollary}$ to the holomorphic self-mapping of $\mathbb D$ defined by
$$g(z)=\frac{f(\xi z)}{f(\xi)}, \qquad \forall\, z\in\mathbb D$$ yields the inequality $$|f'(\xi)|\geq\frac{2}{1+|f'(0)|}$$ for any point $\xi$ in the interior of $C$. Consequently, the desired result follows.
\end{proof}

\bibliographystyle{amsplain}

\end{document}